\documentclass[11pt]{amsart}
\usepackage[margin=3cm]{geometry} 
\usepackage{amsmath}
\usepackage{amssymb}
\usepackage{amsthm}
\usepackage{float}
\usepackage{ulem}
\usepackage{xcolor}
\usepackage[utf8x]{inputenc}
\usepackage{comment}
\usepackage{framed}
\usepackage{caption}

\newcommand{\abs}[1]{\left\vert#1\right\vert}

\newcommand{\eps}{\varepsilon}

\DeclareMathAlphabet{\mathbbb}{U}{bbold}{m}{n}

\newtheorem{thm}{Theorem}[section]

\newtheorem{lem}[thm]{Lemma}

\theoremstyle{definition}
\newtheorem{defn}[thm]{Definition}
\newtheorem{ex}[thm]{Example}
\theoremstyle{remark}

\numberwithin{equation}{section}
\usepackage{graphicx}

\usepackage{enumerate}

\newcommand{\threearray}[3]{\begin{aligned}&{#1}\\&{#2}\\&{#3}\end{aligned}}

\newcommand{\matrice}{\begin{pmatrix}}
\newcommand{\ok}{\end{pmatrix}}

\newcommand{\bd}{\partial}

\begin{document}
\title[]{Maximisers for the second Neumann eigenvalue of annuli in two-dimensional space forms}

\author[Provenzano]{Luigi Provenzano}
\address{Dipartimento di Matematica e Informatica, Università degli Studi di Cagliari, Via Ospedale 72 - 09124 Cagliari, Italy, e-mail: {\sf luigi.provenzano@unica.it}.}
\author[Savo]{Alessandro Savo}
\address{Dipartimento di Scienze di Base e Applicate per l'Ingegneria, Universit\`a di Roma ``La Sapienza'', Via Scarpa 12 - 00161 Roma, Italy, e-mail: {\sf alessandro.savo@uniroma1.it}.}

\begin{abstract}
We prove a sharp Szeg\"o-Weinberger-type inequality for balanced doubly connected domains in $\mathbb S^2$, $\mathbb R^2$ and $\mathbb H^2$.
More precisely, the first positive Neumann eigenvalue is bounded above by that of the corresponding rotationally symmetric annulus, determined by the area of the domain and the area of the smallest component of the complement. Equality holds only for the rotationally symmetric annulus. We also show that the balancing assumption is essential: if we drop it, the inequality fails in general.
\end{abstract}


\thanks{The authors acknowledge the support of the INdAM GNSAGA group.}

\maketitle

\section{Introduction and statement of the main results}

Throughout this paper, we denote by $M_K$ the two-dimensional simply connected space form of constant curvature $K=-1$, $0$ or $1$: $M_{-1}=\mathbb H^2$ (hyperbolic plane), $M_0=\mathbb R^2$ (Euclidean plane) and $M_1=\mathbb S^2$ (round two-sphere).
Let $\Omega$ be a smooth bounded domain of $M_K$. Consider
\begin{equation}\label{neumann}
\begin{cases}
\Delta u=\mu u\,, & {\rm in\ }\Omega\,,\\
\partial_Nu=0\,, & {\rm on\ }\partial\Omega,
\end{cases}
\end{equation}
where $N$ is the unit inner normal. The eigenvalues are denoted by
$$
0=\mu_1(\Omega)<\mu_2(\Omega)\leq\cdots\leq\mu_j(\Omega)\leq\cdots\nearrow+\infty.
$$

The following result is now well-known:

\begin{thm}\label{main}
Let $\Omega$ be a simply connected domain of $M_K$. Then
\begin{equation}\label{mainresult}
\mu_2(\Omega)\leq\mu_2(\Omega^{\star}),
\end{equation}
where $\Omega^{\star}$ is a geodesic disk with $|\Omega|=|\Omega^{\star}|$. Equality holds if and only if $\Omega$ is congruent to  $\Omega^{\star}$.
\end{thm}
When $K=0$, i.e., for $\mathbb R^2$, Theorem \ref{main} was proved by Szeg\"o \cite{szego_1} in 1954. For $\mathbb H^2$, Theorem \ref{main} was proved by Bandle \cite{bandle2} in 1972. In the same paper, Theorem \ref{main} is proved also for simply connected domains $\Omega$ in $\mathbb S^2$ with $|\Omega|\leq 2\pi$. In higher dimensions, Theorem \ref{main} holds for domains in $\mathbb R^n$ with no  topological assumptions thanks to an argument of Weinberger \cite{weinberger}. It holds for domains in $\mathbb H^n$ and in the hemisphere $\mathbb S^n_+$, as well, again with no topological constraints, see \cite{ash_beng}. The spherical case is much more involved, and, quite surprisingly, the inequality does not hold in general if one removes the hypothesis of simple connectivity, see \cite{bucur_gafa} for a counterexample in $\mathbb S^2$ with five boundary components. In \cite{SP26} we completed the study of the spherical case proving Theorem \ref{main}, and showing that there exist counterexamples in $\mathbb S^2$ with just two boundary components (i.e., spherical annuli).

\smallskip

In view of the counterexample in \cite{SP26}, it is natural to ask whether an isoperimetric inequality holds if we restrict to bounded {\it doubly connected domains} in $M_K$, that is, domains diffeomorphic to  $\mathbb S^1\times(-1,1)$. Results in this spirit already exist in the literature. In \cite{verma_santhanam} the authors consider the following problem: maximise the second Neumann eigenvalue of $D\setminus B$ with respect to $D$, where $B$ is a fixed ball in $\mathbb R^n$ centered at $p$ and $D$ is a domain symmetric with respect to $p$, satisfying $\overline B\subset D$. The maximisers are concentric shells. The result holds also for non-compact rank-1 symmetric spaces. In \cite{ABD_SIMA} the authors consider domains of the form $\Omega_{out}\setminus\overline{\Omega_{in}}$, where $\Omega_{out}$, $\Omega_{in}$ are domains in $\mathbb R^n$ and $\overline\Omega_{in}\subset\Omega_{out}$. Let  $B_1,B_2$ concentric balls such that $B_1$ is contained in $\Omega_{in}$ and $|B_2\setminus\overline{B_1}|=|\Omega_{out}\setminus\overline{\Omega_{in}}|$. The authors prove that $\mu_2(\Omega_{out}\setminus\overline{\Omega_{in}})\leq\mu_2(B_2\setminus\overline{B_1})$, with equality if and only if $\Omega_{out}\setminus\overline{\Omega_{in}}=B_2\setminus\overline{B_1}$. In \cite{anoop_verma} the result is extended to simply connected space forms. However, a genuine-type of isoperimetric inequality would involve the prescription of the total area of the domain $\Omega_{out}\setminus\overline{\Omega_{in}}$ and the area of $\Omega_{in}$, which is a question asked in \cite{ABD_SIMA}. The kind of statement in \cite{ABD_SIMA},  \cite{anoop_verma} and \cite{verma_santhanam} is constrained by the use of the method of Weinberger \cite{weinberger}.

\smallskip In this paper we are able to prove an isoperimetric inequality for {\it balanced} doubly connected domains in $M_K$ (see Definition 1.2 below). More precisely, we prove that the first positive Neumann eigenvalue is bounded above by that of the corresponding rotationally symmetric annulus uniquely determined (up to congruence) by the area of the domain, and the area of the smallest connected component of the complement. To the best of our knowledge, no sharp Szeg\"o–Weinberger-type inequality is available already in dimension $2$, in particular in $\mathbb R^2$ for doubly connected (symmetric/balanced) domains under the simultaneous prescription of the volume of the domain and of its hole. In order to state our main result, we recall some definitions:

 \begin{defn}\label{defn_sym}
Let $\Omega\subset M_K$ be a bounded doubly connected domain. We say that it is {\it balanced} if it is isometric to the flat cylinder $C_L=(-L,L)\times\mathbb S^1$ with the conformal metric $\rho^2 (dz^2+d\theta^2)$ and $\int_{\mathbb S^1}\rho^2(z,\theta)e^{i\theta}d\theta=0$ for all $z\in(-L,L)$.
\end{defn}
For a bounded doubly connected  domain $\Omega$ we set
$$
\overline\Omega^c=\omega_1\cup\omega_2,
$$ 
where $\omega_1,\omega_2$ are disjoint domains. Our convention will be that $\omega_1$ has the smallest area:
$$
|\omega_1|\leq|\omega_2|,
$$
Note that for $M_0=\mathbb R^2$ or $M_{-1}=\mathbb H^2$, $\omega_2$ is unbounded, hence $|\omega_2|=\infty$ and the inequality is strict. Moreover, $\omega_1$ is always simply connected.

\begin{ex}
A simple class of balanced domains, besides round annuli, is provided by doubly connected planar domains which are symmetric with respect to both coordinate axes. Let $(x,y)$ denote the usual Cartesian coordinates in $\mathbb R^2$ and let $\Omega$ be such a domain. One then finds a conformal map $\Phi:\Omega\to C_L$ with conformal factor $\rho$ satisfying
$$
\rho^2(z,\theta+\pi)=\rho^2(z,\theta).
$$
Consequently $\int_{\mathbb S^1}\rho^2(z,\theta)e^{i\theta}\,d\theta=0$
for every $z\in(-L,L)$. The same observation applies to $\mathbb S^2$ and $\mathbb H^2$. 
\end{ex}

\begin{defn}
Let $p\in M_K$ and let $(r,\theta)\in[0,+\infty)\times[0,2\pi)$ be the polar coordinates based at $p$ (when $M_K=\mathbb S^2$, it is understood that $r\leq\pi$). Let $0<r_1<r_2$. We define
\begin{equation}\label{A12}
A_{r_1,r_2}:=\{(r,\theta):r_1<r<r_2\}
\end{equation}
the annulus of $M_K$ of radii $r_1,r_2$. Without loss of generality, in the case  of $\mathbb S^2$, we assume that
$$
r_1<\pi-r_2.
$$
\end{defn}
Let $A_{r_1,r_2}$ be an annulus in $M_K$ centered at some $p\in M_K$. Denote by $d_1$ the geodesic disk in $M_K$ centered at $p$ of radius $r_1$. We are now in position to give the following:

\begin{defn}\label{omstar}
Let $\Omega$ be a bounded doubly connected domain in $M_K$. Then, up to isometries, there exists a unique annulus $A_{r_1,r_2}$ in $M_K$ such that $|\Omega|=|A_{r_1,r_2}|$ and $|\omega_1|=|d_1|$. For any bounded doubly connected domain $\Omega$ we set
$$
\Omega^\star:=A_{r_1,r_2}.
$$
\end{defn}

Here is the main result:

\begin{thm}\label{main_isop_annuli}
Let $\Omega$ be a balanced, bounded doubly connected domain in $M_K$. Then
\begin{equation}\label{main_formula_annuli}
\mu_2(\Omega)\leq\mu_2(\Omega^\star).
\end{equation}
Equality holds if and only if $\Omega$ is congruent to $\Omega^\star$.
\end{thm}

The proof follows the same arguments as in Theorem \ref{main} in \cite{SP26}, replacing, in the present case,  the Green function with a suitable {\it capacitary function} of the doubly connected domain $\Omega$ (that is, a function which is harmonic in $\Omega$ and locally  constant on the boundary). Its level lines provide a smooth foliation of $\Omega$; we perform a change a gauge and look at test functions which are real and constant on these level lines. We need the requirement that the domain is balanced since the conformal group of the annulus is poor and does not allow for a fixed point argument guaranteeing the existence of a test function orthogonal to the zero eigenfunction (in the simply connected case of \cite{SP26} we used the conformal group of the disk, which is richer).

\smallskip

We remark that the  balancing hypothesis in Theorem \ref{main_isop_annuli} cannot be dropped. For $\mathbb R^2$, this has been shown  in \cite{ABD_SIMA}, building on the papers \cite{kuttler_wave,yee_wave}; the counterexample in [1] is given by eccentric annuli, in the following sense. Let $B$ be the unit disk in $\mathbb R^2$ centered at $0$ and let $B_r(p)$ be the disk of radius $r$ centered at $p\in B$ with $r<1$ and $|p|<1-r$. When $|p|=r=\frac 14$, one has $\mu_2(B\setminus\overline B_r(0))<\mu_2(B\setminus\overline B_r(p))$. Note that the domain $B\setminus\overline B_r(p)$ is not balanced. Hence, inequality \eqref{main_formula_annuli} cannot hold in general without some ``symmetry'' or ``balancing''  assumption. The same construction can be reproduced in $\mathbb S^2$ and $\mathbb H^2$: just take the following domains: $\Omega_\eps:=B_\eps(q)\setminus \overline B_{r\eps}(q')$ and $\Omega_\eps^\star:=B_\eps(q)\setminus \overline B_{r\eps}(q)$, where $q$ is some fixed point, $q'$ is a point at distance $\frac\eps 4$ from $q$, and $r=\frac 14$. Here we are denoting by $B_s(q),B_s(q')$ the geodesic disks centered respectively at $q,q'$ of radius $s>0$.  After rescaling the metric by $\eps^{-2}$, the Neumann eigenvalues of $\Omega_\eps$ and $\Omega_\eps^\star$ multiplied by $\eps^2$ converge respectively to those of the  Euclidean domains $B\setminus\overline B_r(p)$ and $B\setminus\overline B_r(0)$ discussed above (and in \cite{ABD_SIMA}), hence the strict Euclidean inequality proved in \cite{ABD_SIMA} persists for all sufficiently small $\eps$.

\smallskip

We also remark that the same technique applies, in a simpler way, to the first eigenvalue of the magnetic Laplacian with a closed potential $1$-form of non-integral flux on doubly connected domains of $M_K$, under magnetic Neumann boundary conditions; see \cite{CPS22} for basic information on this problem. We will address the magnetic counterpart in a forthcoming paper.

\smallskip

{\bf Organization of the paper.} The paper is organised as follows: in Section \ref{sec:2} we recall a few useful facts on the magnetic Laplacian and the change of gauge. In Section \ref{sec_radialspectrum} we introduce the notion of radial spectrum: it is a variational spectrum associated with functions that are real and constant on the level lines of the capacitary function. It will be the main ingredient in the proof of the main result. In Section \ref{sec:4} we describe the proof of Theorem \ref{main_isop_annuli}, which is based on Theorems \ref{three}, \ref{four} and \ref{barycenter} proved respectively in Sections \ref{sub_3}, \ref{sub_4} and \ref{sec_barycenter}.

\section{A few useful facts on the magnetic Laplacian}\label{sec:2}

Throughout the paper, $\Omega$ will always denote a bounded doubly connected domain in $M_K$.

Let $A$ be a smooth closed $1$-form in $\Omega$. We denote by $\nu$ the {\it flux} of $A$ along $c$, namely 
$$
\nu\doteq\frac{1}{2\pi}\oint_{c}A
$$
where $c$ is a simple, closed, oriented curve in $\Omega$ homotopic to any of the two boundary components of $\Omega$. The choice of the orientation will not play a role in our estimates; however we understand that $c$ is travelled once in the counterclockwise direction. Let $d^A$ denote the magnetic differential: $d^Au=du-iuA$ (if $A=0$ it is the standard differential), and let $\delta^A$ its formal $L^2$-adjoint  (magnetic co-differential). Then $\delta F=-{\rm div}F$ if $F$ is a $1$-form.

The magnetic Laplacian is defined as $\Delta_A\doteq\delta^Ad^A$. If $A=0$ then $\Delta_A=\Delta$ is the usual Laplacian.  We consider the Neumann problem for the magnetic Laplacian
$$
\begin{cases}
\Delta_Au=\lambda u\,, & {\rm in\ }\Omega\\
d^Au(N)=0\,, & {\rm on\ }\partial \Omega.
\end{cases}
$$
Here $N$ is the unit inner normal to $\partial \Omega$.  The spectrum is discrete, made of non-negative eigenvalues of finite multiplicity (see e.g., \cite{CPS22}):
$$
0\leq\lambda_1(\Omega,A)\leq\lambda_2(\Omega,A)\leq\cdots\leq\lambda_j(\Omega,A)\leq\cdots\nearrow+\infty
$$
The eigenvalues are characterized by
\begin{equation}\label{minmaxA}
\lambda_j(\Omega,A)=\min_{{\substack{U\subset H^1(\Omega)\\{\rm dim}U=j}}}\max_{0\ne u\in U}\frac{\int_{\Omega}|d^Au|^2}{\int_{\Omega}|u|^2}.
\end{equation}

When $A=0$ we have the usual Neumann problem \eqref{neumann} for the Laplacian on $\Omega$: on $\partial\Omega$
$$
du(N)=0\iff\partial_Nu=0
$$
Recall that the usual Neumann eigenvalues (for the Laplacian) are characterized by
$$
\mu_j(\Omega)=\min_{{\substack{U\subset H^1(\Omega)\\{\rm dim}U=j}}}\max_{0\ne u\in U}\frac{\int_{\Omega}|du|^2}{\int_{\Omega}|u|^2}.
$$

\subsection{Orientation and Hodge $\star$}
On $\mathbb R^2$ we consider the standard orientation given by the positively oriented basis $(e_1,e_2)$. The standard orientation of $\mathbb H^2$ is the same as $\mathbb R^2$ when considering the Poincaré disk model. Consider $\mathbb S^2\subset\mathbb R^3$, and let $X=(x_1,x_2,x_3)\in\mathbb S^2$ be the unit  vector field normal to $\mathbb S^2$ (pointing outwards). Let $p\in\mathbb S$. Let $(e_1,e_2)$ be an orthonormal basis of $T_p\mathbb S^2$. We say that it is positively oriented if $(e_1,e_2,X)$ is a positively oriented basis of $\mathbb R^3$.

Next, we introduce the Hodge-star operator $\star$ acting on tangent vectors, which is uniquely determined by the following condition: if $v$ is a unit vector in $T_pM_K$, then  $(v,\star v)$ is a positively oriented basis of $T_pM_K$. In other words, $\star$ is simply rotation by $\pi/2$, so that $\star$ is an isometry satisfying $\star^2=-I$. We can think of $\star$ as acting on a smooth $1$-form $A$ by the rule:
$$
\star A(v)=-A(\star v);
$$
In particular, for a smooth function $\psi$ one sees that $\star d\psi$ is the dual $1$-form of the vector field normally denoted $\nabla^{\perp}\psi$.

\subsection{The capacitary function and the magnetic potential}

Let $\Omega\subset M_K$ be a bounded doubly connected domain, recall that $\overline\Omega^c=\omega_1\cup\omega_2$ where $\omega_1$ is the component of the complement having smallest area. Then 
$$
\partial\Omega=\Gamma_1\cup\Gamma_2
$$
where $\Gamma_1,\Gamma_2$ are disjoint topological circles. Our convention is that $\Gamma_1=\bd\omega_1$.
For a bounded doubly connected domain, the role played by the Green function in \cite{SP26} will be played by the {\it capacitary function} $\psi$, which is the  unique solution of
\begin{equation}\label{green}
\begin{cases}
\Delta\psi=0\,, & {\rm in\ }\Omega\,,\\
\psi=0\,,& {\rm on\ }\Gamma_1\,,\\
\psi=c\,, & {\rm on\ \Gamma_2}
\end{cases}
\end{equation}
where $c>0$ is uniquely determined by imposing that the flux of $\star d\psi$ along $\Gamma_1$ is equal to $1$, which in turn translates into the equality $\frac{1}{2\pi}\int_{\Gamma_1}\partial_N\psi=1$ (see below). Note that $\psi$ is positive, smooth and harmonic in $\Omega$. By $C_L$ we denote the flat cylinder $(-L,L)\times\mathbb S^1$. Recall that any bounded doubly connected domain $\Omega$ is conformally equivalent to $C_L$ for a unique $L>0$ called the {\it conformal modulus} of $\Omega$. We will show that
$$
c=2L.
$$
We denote by $(z,\theta)$ the coordinates on $C_L$. The standard orientation on $C_L$ is such that $\star dz=-d\theta$.

\begin{lem}\label{lemma_conformal} We have:
\begin{enumerate}[i)]
\item The capacitary function $\psi_L$ of the flat cylinder $C_L$, chosen so that $\frac{1}{2\pi}\int_{\mathbb S^1}\partial_z\psi_L(-L,\theta)d\theta=1$ is given by:
$$
\psi_L(z,\theta)=L+z.
$$
\item The capacitary function is conformally invariant: if $\Omega$ and $\Omega'$ are bounded doubly connected domains and $\Phi:\Omega'\to\Omega$ is a conformal diffeomorphism, and $\psi$ is the capacitary function of $\Omega$, then $\psi\circ\Phi=\Phi^{\star}\psi$ is the capacitary function of $\Omega'$.

\item If $\Omega$ has conformal modulus $L$, then in \eqref{green} we have $c=2L$.

\item $\psi$ has no critical points in $\overline\Omega$: if $\Phi:\Omega\to C_L$ is a conformal map, then $\psi=\Phi^{\star}\psi_L$, and $\psi_L$ has no critical points in $\overline{C_L}$.
\end{enumerate}
\end{lem}

We now take, as a magnetic potential on $\Omega$, the $1$-form:
\begin{equation}\label{ap}
A=\star d\psi.
\end{equation}

\begin{lem} We have:
\begin{enumerate}[i)] 
\item The $1$-form $A$ is smooth, closed and co-closed, hence harmonic in $\Omega$.
\item The flux of $A$ around $\Gamma_1$ (oriented as $\star N$) is equal to $1$. 
\item  On $C_L$ we have $A=-d\theta$.
\end{enumerate}
\end{lem}
\begin{proof}
The proof is immediate. Since $\psi$ is harmonic in $\Omega$ it follows by direct computation that $\star d\psi$ is closed and co-closed. Moreover
$$
\dfrac {1}{2\pi}\oint_{\Gamma_1}A=\frac 1{2\pi}\int_{\Gamma_1}\star d\psi(\star N)=\frac{1}{2\pi}\int_{\Gamma_1}d\psi(N)=\frac{1}{2\pi}\int_{\Gamma_1}\partial_N\psi=1.
$$ 
Point $iii)$ is a direct computation.
\end{proof}

\subsection{Gauge invariance} 

Fix a base point $x_0\in\Omega$ and define a (multivalued) function on $\Omega$ as follows:
$$
\Theta(x):=\int_{c_x}A
$$
where $c_x$ is any curve joining $x_0$ to $x$. Since $A$ is closed, and the flux around a loop in $\Omega$ is an integer, one sees that choosing another such curve $c'_{x}$ one gets that 
$$
\int_{c_x}A-\int_{c_x'}A\in 2\pi\mathbb Z
$$
This means that the function 
$$
e^{i\Theta(x)}
$$
is well-defined and smooth in $\Omega$. Since $A$ is closed and has integer flux around any loop in $\Omega$, from the well-known result in \cite{Sh} we deduce that the spectrum of the magnetic Laplacian with potential $1$-form $A$ coincides with the usual Neumann spectrum, because the operators $\Delta_A$ and $\Delta$ are unitarily equivalent:
$$
\Delta_A=e^{i\Theta}\Delta e^{-i\Theta}.
$$
We summarize as follows.

\begin{lem}\label{one} \begin{enumerate}[i)]
\item For all  $j\in\mathbb N$:
$$
\lambda_j(\Omega,A)=\mu_j(\Omega).
$$
In particular, if $u$ is an eigenfunction of $\Delta_A$ then $e^{-i\Theta}u$ is an eigenfunction of $\Delta$, associated to the same eigenvalue. Vice versa, if $v$ is an eigenfunction of $\Delta$ then $e^{i\Theta}v$ is an eigenfunction of $\Delta_{A}$, associated to the same eigenvalue.
\item One has $\lambda_1(\Omega,A)=0$ with eigenspace spanned by the function $e^{i\Theta}$.
\end{enumerate}
\end{lem}

\section{The radial spectrum}\label{sec_radialspectrum}
Let $\Omega$ be a bounded doubly connected  domain in $M_K$ with conformal modulus $L$. Let 
$$
S\doteq |\Omega|
$$ be the area of $\Omega$. Let  $\psi$ be the capacitary function of $\Omega$ as in \eqref{green}. We introduce the function space:
$$
\mathcal R(\Omega)=\{u{\rm\ real\ valued}: u=g\circ\psi, \quad g\in H^1(0,2L)\}\subset H^1(\Omega).
$$
Thus, $\mathcal R(\Omega)$ consists of all functions in $H^1(\Omega)$ which are constant on the level curves of $\psi$. We set:
\begin{equation}\label{K1-u}
\kappa_1(\Omega,A)=\min_{0\ne u\in\mathcal R(\Omega)}\frac{\int_{\Omega}|d^{A}u|^2}{\int_{\Omega}|u|^2}.
\end{equation}
We call $\kappa_1(\Omega,A)$ the first {\it radial eigenvalue} of $\Omega$ for the potential $A$. Note that $\kappa_1(\Omega,A)$ {\it does not need} to be an eigenvalue of $\Delta_{A}$ (hence of $\Delta$) in $\Omega$.  

Define a function $G:(0, S)\to (0,\infty)$ by:
\begin{equation}\label{Gp}
G(a)\doteq\int_{\psi=\beta(a)}\dfrac{1}{\abs{d\psi}}
\end{equation}
where $\beta(a)$ is such that the super level set $\{\psi>\beta(a)\}$ has area $a$. 
Since $\psi$ has no critical points in $\bar\Omega$, 
we remark that $G$ is bounded above and below by positive constants.

The next result characterizes $\kappa_1(\Omega,A)$ as the minimizer of a one-dimensional Rayleigh quotient associated with a Sturm-Liouville problem.

\begin{lem}\label{two} We have the following:
\begin{enumerate}[i)]
\item 
\begin{equation}\label{K1}
\kappa_1(\Omega,A)=2\pi\min_{0\ne f\in H^1(0,S)}\frac{\int_0^{ S}\left(G(a)f'(a)^2+\frac{f(a)^2}{G(a)}\right)da}{\int_0^Sf^2(a)da}.
\end{equation}
\item $\kappa_1(\Omega,A)$ is the first eigenvalue of the  following Sturm-Liouville problem in $(0,S)$:
\begin{equation}\label{SL1}
\begin{cases}
-(Gf')'+\frac{1}{G}f=\frac{\kappa}{2\pi} f\,, & {\rm in\ }(0,S)\,,\\
f'(0)=f'(S)=0.
\end{cases}
\end{equation}
\item The Sturm-Liouville problem in ii) is regular.
\end{enumerate}
\end{lem}

\begin{proof}  Let $u=g\circ\psi$. One has:
$$
d^Au=(g'\circ\psi)d\psi-i(g\circ\psi)A;
$$
since $A$ and $d\psi$ are pointwise orthogonal by definition \eqref{ap}:
$$
\begin{aligned}
\abs{d^Au}^2&=(g'\circ\psi)^2\abs{d\psi}^2+(g\circ\psi)^2\abs{A}^2\\
&=\Big((g'\circ\psi)^2+(g\circ\psi)^2\Big)\abs{d\psi}^2.
\end{aligned}
$$
By the coarea formula
\begin{equation}\label{u-g}
\begin{aligned}
\int_{\Omega}\abs{d^Au}^2&=\int_0^{2L}\Big(g'(t)^2+g(t)^2\Big)\int_{\psi=t}\abs{d\psi}\,dt\\
&=2\pi\int_0^{2L}\Big(g'(t)^2+g(t)^2\Big)dt.
\end{aligned}
\end{equation}
The last equality follows from
$$
\int_{\psi=t}\abs{d\psi}=\int_{\psi=t}\partial_N\psi=2\pi.
$$
We now change variable as follows. Write:
$$
\alpha(t)\doteq\abs{\{\psi>t\}}=\int_t^{2L}\left(\int_{\psi=s}\dfrac{1}{\abs{d\psi}}\right)ds
$$
so that
$$
\alpha'(t)=-\int_{\psi=t}\dfrac{1}{\abs{d\psi}}.
$$
Since $\psi$ has no critical points in $\overline\Omega$, then $\abs{d\psi}\geq c>0$. This implies that $\alpha: (0,2L)\to (0,S)$ is smooth, strictly decreasing, and admits a smooth inverse $\beta:(0,S)\to (0,2L)$. We set:
$$
t=\beta(a), \quad g\circ\beta=f,
$$
hence $g(t)=f(a)$. Since $\beta(\alpha(t))=t$ we have
$
\beta'(\alpha(t))\alpha'(t)=1, 
$
which means
$$
\beta'(a)=\dfrac{1}{\alpha'(\beta (a))}=-\dfrac{1}{\int_{\psi=\beta(a)}\frac{1}{\abs{d\psi}}}.
$$
Defining the function $G:(0,S)\to(0,\infty)$ as in \eqref{Gp} by
\begin{equation*}
G(a)=\int_{\psi=\beta(a)}\frac{1}{\abs{d\psi}},
\end{equation*}
we conclude that 
$$
\beta'(a)=-\dfrac{1}{G(a)}.
$$
Now 
$$
\threearray
{g'(t)=g'(\beta(a))=(g\circ\beta)'(a)\cdot\dfrac{1}{\beta'(a)}=\dfrac{f'(a)}{\beta'(a)}=-f'(a)G(a),}
{g(t)=f(a),}
{dt=\beta'(a)da=-\dfrac{da}{G(a)}.}
$$
We conclude
$$
\begin{aligned}
\int_{\Omega}\abs{d^Au}^2&=2\pi\int_0^{2L}\Big(g'(t)^2+g(t)^2\Big)dt\\
&=2\pi\int_0^{ S}\left(G(a)f'(a)^2+\frac{f(a)^2}{G(a)}\right)da.
\end{aligned}
$$
On the other hand, if $u=g\circ\psi$:
$$
\begin{aligned}
\int_{\Omega}u^2&=\int_0^{2L}g(t)^2\int_{\psi=t}\dfrac{1}{\abs{d\psi}}\,dt\\
&=\int_0^Sf(a)^2\,da.
\end{aligned}
$$
This proves $i)$. The proof of $ii)$ is standard Sturm-Liouville theory and follows directly from $i)$. $iii)$ is a consequence of the fact that the weight $G$ is smooth and bounded above and below by positive constants. 
\end{proof}

We remark that we have defined $\kappa_1(\Omega,A)$ as the minimum of a Rayleigh quotient over a certain subspace of $H^1(\Omega)$. Then $\kappa_1(\Omega,A)$ turns out to be the first eigenvalue of a suitable Sturm-Liouville problem \eqref{SL1}. We can define a whole sequence of eigenvalues $\{\kappa_j(\Omega,A)\}_{j=1}^{\infty}$ via the min-max procedure, which then coincides with the spectrum of \eqref{SL1}. We call this spectrum the {\it radial spectrum}. Note that these need not  be actual eigenvalues of $\Delta_{A}$ in $\Omega$. However, as we shall see, sometimes they are. In this paper we just work with $\kappa_1(\Omega,A)$.

\subsection{The case of the cylinder}\label{sub:disk}

On the cylinder $C_L=(-L,L)\times \mathbb S^1$ the capacitary function is
$$
\psi_L(z,\theta)=L+z
$$
hence $\{\psi_L=t\}$ is the unit circle $\mathbb S^1\times\{t-L\}$, so that
$$
\abs{\{\psi_L=t\}}=2\pi, \quad \alpha(t)=\abs{\{\psi_L>t\}}=2\pi(2L-t).
$$
Hence, if $\beta(a)=\alpha^{-1}(a)$, inverting $a=2\pi(2L-t)$ gives
$$
\beta(a)=2L-\frac{a}{2\pi}.
$$
Now $d\psi_L=dz$, hence 
$\abs{d\psi_L}=1$.
Finally 
$$
G_L(a):=\int_{\psi_L=\beta(a)}\dfrac{1}{|d\psi_L|}=\abs{\{\psi_L=\beta(a)\}}\cdot\dfrac{1}{\abs {d\psi_L}}=2\pi.
$$

\subsection{The case of an annulus}\label{sub:sphericalcup}
Let $\Omega$ be a bounded doubly connected domain in $M_K$, with boundary $\partial\Omega=\Gamma_1\cup\Gamma_2$, where $\Gamma_i$ bound the domain $\omega_i$ in the complement of $\Omega$. Recall that by our convention  $|\omega_1|\leq|\omega_2|$ so that $|\omega_1|<\infty$. Let now $\Omega^\star=A_{r_1,r_2}\subset M_K$ be the symmetric annulus centered at some $p\in M_K$ of radii $0<r_1<r_2$ such that, if $d_1$ is a disk of radius $r_1$ centered at $p$, then $|\Omega^\star|=|\Omega|=S$ and $|d_1|=|\omega_1|$, as in Definition \ref{omstar}. Let   
$$
h_1:=|d_1|=|\omega_1|.
$$
Let $\psi^\star$ denote the  capacitary function of $\Omega^\star$ as in \eqref{green},  $A^\star$ the associated magnetic potential, and $G^\star$ the function defined in \eqref{Gp} when $\Omega=\Omega^\star$ and $\psi=\psi^\star$. In polar coordinates $(r,\theta)$, where $r$ is the geodesic distance from the center of the annulus, the capacitary function $\psi^\star$ is an explicit radial function, hence, proceeding as in the case of the cylinder $C_L$, we see that $$
G^\star(a)=\frac{(h_1+S-a)(4\pi-K(h_1+S-a))}{2\pi}.
$$ However, the same fact will also drop from the following isoperimetric inequality, which will be needed later. 
Let $G:(0,S)\to (0,\infty)$ be defined as in \eqref{Gp} in terms of the capacitary function $\psi$ of $\Omega$. Then we have:

\begin{lem}\label{lem31} One has, on $(0,S)$:
\begin{equation}\label{ineq*}
G^\star\leq G 
\end{equation}
 with equality a.e. if and only if $\Omega$ is congruent to $\Omega^{\star}$. Moreover
 \begin{equation}\label{ineq2*}
G^\star(a)=\frac{(h_1+S-a)(4\pi-K(h_1+S-a))}{2\pi}.
 \end{equation}
for all $a\in (0,S)$.
\end{lem}
\begin{proof}
 For all $t\in (0,2L)$ we have
\begin{equation}\label{pre-isoperimetric}
\begin{aligned}
\abs{\{\psi=t\}}&=\int_{\psi=t}1\\
&=\int_{\psi=t}\abs{d\psi}^{\frac 12}\cdot\dfrac{1}{\abs{d\psi}^{\frac 12}}\\
&\leq \Big(\int_{\psi=t}\abs{d\psi}\Big)^{\frac 12}\cdot 
\Big(\int_{\psi=t}\dfrac{1}{\abs{d\psi}}\Big)^{\frac 12}\\
&=\Big(2\pi\int_{\psi=t}\dfrac{1}{\abs{d\psi}}\Big)^{\frac 12}
\end{aligned}
\end{equation}
with equality if and only if $\abs{d\psi}$ is constant on $\psi=t$. Recall that $\alpha(t)\doteq |\{\psi>t\}|$ and $\beta$ is its inverse. Taking $t=\beta(a)$ and passing to the variable $a$, \eqref{pre-isoperimetric} reads
$$
\abs{\{\psi=\beta(a)\}}^2\leq 2\pi G(a).
$$
We use the well-known isoperimetric inequality for bounded domains in $M_K$: $\ell^2\geq a(4\pi-Ka)$, where $a$ is the area and $\ell$ is the boundary length. We refer e.g., to \cite{chavel_isop} for a proof. Let $D_a$ be the simply connected component of
$M_K\setminus\{\psi=\beta(a)\}$ containing $\omega_1$. Since
$a=|{\psi>\beta(a)}|$, we have $|D_a|=h_1+S-a$. Hence $|\partial D_a|^2\geq |D_a|(4\pi-K|D_a|)$. Therefore
$$
\abs{\{\psi=\beta(a)\}}^2\geq (h_1+S-a)(4\pi-K(h_1+S-a)),
$$
from which we get \eqref{ineq*} 
with equality if and only if each level set $\{\psi=t\}$ is a circle and $\abs{d\psi}$ is constant on $\psi=t$. This last condition implies that the level sets are parallel to one another, hence $\psi$ must be a radial function; hence $\Omega$ is congruent to $\Omega^\star$. Finally, if $\Omega$ is congruent to $\Omega^\star$  we have equality everywhere, in particular we have \eqref{ineq2*} for all $a\in (0,S)$.
\end{proof}

\section{Proof of Theorem \ref{main_isop_annuli}}\label{sec:4}
 Theorem \ref{main_isop_annuli} will follow by combining gauge invariance (Lemma \ref{one}) with three results that we state here below, namely Theorems \ref{three}, \ref{four} and \ref{barycenter}.

\medskip

The first result is a comparison between  $\kappa_1(\Omega,A)$ and $\kappa_1(\Omega^\star,A^\star)$.

\begin{thm}\label{three} We have
$$
\kappa_1(\Omega,A)\leq \kappa_1(\Omega^{\star}, A^{\star}),
$$
 with equality if and only if $\Omega$ is congruent to $\Omega^{\star}$.
\end{thm}

\medskip

We recall that a radial eigenvalue $\kappa_1(\Omega,A)$ is not necessarily a Neumann eigenvalue of $\Delta_A$. The next result states that the first radial eigenvalue for $\Omega^{\star}$ is actually a Neumann eigenvalue of $\Delta_{A^\star}$ (hence of $\Delta$), more precisely, it is the second eigenvalue.

\begin{thm}\label{four} One has $\kappa_1(\Omega^{\star}, A^\star)=\lambda_2(\Omega^{\star},A^\star)$.
%
\end{thm}

A radial eigenfunction, as just said, does not need to be an eigenfunction of $\Delta_A$. Moreover it does not need to be orthogonal to the first eigenfunction $e^{i\Theta}$ of $\Delta_A$, so that it cannot be used to estimate from above $\lambda_2(\Omega,A)$ using the min-max principle \eqref{minmaxA}. However, if $\Omega$ is balanced in the sense of Definition \ref{defn_sym}, the orthogonality condition is automatically satisfied.

\begin{thm}\label{barycenter} If $\Omega$ is a balanced, bounded doubly connected domain, then
$$
\lambda_2(\Omega, A)\leq \kappa_1(\Omega,A).
$$
\end{thm}

The proof of Theorem \ref{main_isop_annuli} is achieved by combining Lemma \ref{one} and Theorems \ref{three}, \ref{four} and \ref{barycenter}:
$$
\begin{aligned}
\mu_2(\Omega)&=\lambda_2(\Omega,A)\quad\text{by Lemma \ref{one}}\\
&\leq \kappa_1(\Omega,A)\quad\text{by Theorem \ref{barycenter}}\\
&\leq \kappa_1(\Omega^{\star},A^\star)\quad\text{by Theorem \ref{three}}\\
&=\lambda_2(\Omega^{\star},A^\star)\quad\text{by Theorem \ref{four}}\\
&=\mu_2(\Omega^{\star})\quad\text{by Lemma \ref{one}}.
\end{aligned}
$$
Moreover, if $\mu_2(\Omega)=\mu_2(\Omega^\star)$, then $\kappa_1(\Omega,A)=\kappa_1(\Omega^{\star},A^\star)$, and from Theorem \ref{three} we see that $\Omega$ must be congruent to $\Omega^{\star}$. \qed


\section{Proof of Theorem \ref{three}}\label{sub_3}

The proof of Theorem \ref{three} is a consequence of the isoperimetric inequality (Lemma \ref{lem31}), and a monotonicity result for the first eigenvalue of a weighted Sturm-Liouville problem with respect to the weight. 
\medskip

Let $G:(0,S)\to (0,\infty)$ be a smooth, bounded, positive function, and $\kappa_1(G)$ be defined by
\begin{equation}\label{RayleghG}
\kappa_1(G):=2\pi\min_{0\ne f\in H^1(0,S)}\frac{\int_0^{ S}\left(G(a)f'(a)^2+\frac{f(a)^2}{G(a)}\right)da}{\int_0^{S}f^2(a)da}.
\end{equation}
Note that $\kappa_1(\Omega,A)=\kappa_1(G)$ when $G$ is defined as in \eqref{Gp}.

\begin{lem}\label{32} If $G\geq G^\star$ on $(0,S)$ then
$$
\kappa_1(G^\star)\geq \kappa_1(G).
$$
If $G>G^\star$ on a set of positive measure, the inequality is strict.
\end{lem}
\begin{proof}
For $t\in [0,1]$ define $G_t:(0,S)\to (0,\infty)$ by
$$
G_t=(1-t)G^\star+tG.
$$
and let $f_t$ be the positive, $L^2$-normalised minimiser in \eqref{RayleghG} with $G=G_t$. As in the proof of Lemma \ref{two}, $\kappa_1(G_t)$ is the first eigenvalue of the regular Sturm-Liouville problem \eqref{SL1} (with $G_t$ replacing $G$); as $\kappa_1(G_t)$ is simple, we can apply the Feynman-Hellmann formula (see \cite[VII-\S4, p. 408, formula 4.56]{kato1}) which in our case gives
$$
\dfrac{d}{dt}\kappa_1(G_t)=2\pi\int_0^S\Big(G_t(a)^2f'_t(a)^2-f_t(a)^2\Big)\dfrac{G(a)-G^\star(a)}{G_t(a)^2}\,da.
$$
It will then be enough to show that $G_t^2f_t'^2- f_t^2<0$ on $(0,S)$.  Equivalently, let us fix an arbitrary $t\in [0,1]$ and consider the smooth function $R:(0,S)\to\mathbb R$ defined by
$$
R=\dfrac{G_tf'_t}{f_t}.
$$
It is continuous in $[0,S]$ and, since $f_t'(0)=f_t'(S)=0$ and $f_t(0)>0$, $f_t(S)>0$, we see
\begin{equation}\label{rmzero}
R(0)=R(S)=0.
\end{equation}
It is enough to show that $\abs{R(a)}<1$ for all $a\in (0,S)$. We differentiate $R$ and get:
\begin{equation}\label{rprime}
R'=\dfrac{1-R^2}{G_t}-\frac{\kappa_1(G_t)}{2\pi}
\end{equation}
which implies that if $\abs{R}\geq 1$ then $R'<0$. 

\medskip

{\bf First case.} Assume  that there exists $a_0\in (0,S)$ such that $R(a_0)\geq 1$. Then $R'(a_0)<0$ and one sees easily that $R$ is decreasing (and positive) on $(0,a_0)$. Therefore
 $\lim_{a\to 0}R(a)>1$  which is a contradiction with $R(0)=0$.

\smallskip

{\bf Second case.} Assume that there exists $a_0$ such that $R(a_0)\leq -1$.  Then $R'(a_0)<0$ and $R'<0$ on $(a_0,S)$. Therefore $\lim_{a\to S}R(a)<-1$ which is a contradiction with $R(S)=0$.
The proof is complete.
\end{proof}

Theorem \ref{three} now follows by taking $G$ as in \eqref{Gp} and recalling by Lemma \ref{lem31} that $G\geq G^\star$.
\qed

\section{Proof of Theorem \ref{four}}\label{sub_4}

Let $\Omega^{\star}=A_{r_1,r_2}$ be the symmetric annulus centered at some $p\in M_K$ (see \eqref{A12} for the definition).  We briefly recall the description of the spectrum of the Neumann Laplacian on $\Omega^{\star}$. 

We use polar coordinates $(r,\theta)$ around $p$
 and separate variables. As usual, we find a basis of $L^2(\Omega^{\star})$ of eigenfunctions in the form $u(r,\theta)=v(r)e^{ik\theta}$, where $k\in\mathbb Z$ (see e.g., \cite[\S II.5]{chavel}).  Expressing the  Laplacian $\Delta$ in polar coordinates, we see that $u$ as above is an eigenfunction of the Laplacian on $\Omega^{\ast}$ satisfying the Neumann boundary condition  if and only if the radial part $v(r)$ is an eigenfunction of the following Sturm-Liouville problem:
\begin{equation}\label{SLK}
SL_k:
\begin{cases}
-v''-\frac{s_K'(r)}{s_K(r)}\,v'+\frac{k^2}{s_K^2(r)}v=\mu v\,, & {\rm in\ }(r_1,r_2)\,,\\
 v'(r_1)=v'(r_2)=0,
\end{cases}
\end{equation}
where
\begin{equation}\label{sk}
s_K(r):=\begin{cases}
\sinh(r) & {\rm\ if\ }K=-1\,,\\
r & {\rm\ if\ }K=0\,,\\
\sin(r) & {\rm\ if\ }K=1\,.\\
\end{cases}
\end{equation}

 For each fixed $k\in\mathbb Z$,  problem $SL_k$ has a countable set of eigenvalues, denoted 
$$
\mu_{k1}\leq \mu_{k2}\leq\cdots\leq\mu_{kj}\leq\cdots\nearrow+\infty
$$
with associated eigenfunctions $v_{kj}(r)$, where $j=1,2,\dots$. Note that $\mu_{kj}=\mu_{(-k)j}$ for all $k\in\mathbb Z$, hence we may confine the analysis of the eigenvalues to $k=0,1,...$.
The smallest eigenvalue is $\mu_1(\Omega^{\star})=\mu_{01}=0$ with eigenspace given by the constants. The second eigenvalue (lowest positive eigenvalue) is denoted $\mu_2(\Omega^{\star})$: it could be either $\mu_{02}$ (radial) or  $\mu_{11}$ (angular mode $k=1$). We claim that
\begin{equation}\label{second_N}
\mu_{11}<\mu_{02}
\end{equation}
Assuming \eqref{second_N} holds, we see that $\mu_2(\Omega^{\star})=\mu_{11}$. In particular,  the  corresponding eigenspace is spanned by $e^{i\theta}v_{11}(r)$ and $e^{-i\theta}v_{11}(r)$. 
Using gauge invariance (Lemma \ref{one}), we see that $\lambda_2(\Omega^{\star},A^\star)=\mu_2(\Omega^{\ast})=\mu_{11}$ has multiplicity $2$; since $A^{\star}=-d\theta$ and $e^{i\Theta}=e^{-i\theta}$,  the corresponding eigenspace is  spanned by $v_{11}(r)$ and $e^{-2i\theta}v_{11}(r)$. Then $\lambda_2(\Omega^{\star},A^\star)$ admits an eigenfunction which is real and radial (i.e., constant on the level lines of $\psi^\star$), and this is $v_{11}$; since $v_{11}$ does not change sign, it must be a first radial eigenfunction of $\Delta_{A^\star}$. In conclusion
$$
\kappa_1(\Omega^{\star},A^\star)=\lambda_2(\Omega^{\star},A^\star)
$$
as asserted. It remains to prove \eqref{second_N}.

 Let $v_{02}$ be an eigenfunction associated with $\mu_{02}$. Set
$$
g_{02}:=v_{02}'.
$$
Observing that $(\frac{s_K'}{s_K})'(r)=-\frac{1}{s_K^2(r)}$, one checks that
$$
-g_{02}''(r)-\frac{s_K'(r)}{s_K(r)}g_{02}'(r)+\frac{1}{s_K^2(r)}g_{02}(r)=\mu_{02} g_{02}(r),
$$
and this follows by differentiating the differential equation in \eqref{SLK} with $k=0$, $v=v_{02}$, $\mu=\mu_{02}$. At the endpoints we have
$$
g_{02}(r_1)=g_{02}(r_2)=0.
$$
This means that $\sigma_{02}$ is an eigenvalue of the following Dirichlet problem:
\begin{equation}\label{ODE_annulus_1_2}
\begin{cases}
-v''(r)-\frac{s_K'(r)}{s_K(r)}v'(r)+\frac{1}{s_K^2(r)}v(r)=\lambda v(r)\,, & {\rm in\ }(r_1,r_2)\,,\\
v(r_1)=v(r_2)=0.
\end{cases}
\end{equation}
In particular
$$
\mu_{02}\geq\lambda_{11}
$$
where $\lambda_{11}$ is the first eigenvalue of \eqref{ODE_annulus_1_2}. Recall that $\mu_{11}$ is the first eigenvalue of the Neumann problem:
\begin{equation}
\begin{cases}
-v''(r)-\frac{s_K'(r)}{s_K(r)}v'(r)+\frac{1}{s_K^2(r)}v(r)=\mu v(r)\,, & {\rm in\ }(r_1,r_2)\,,\\
v'(r_1)=v'(r_2)=0.
\end{cases}
\end{equation}
Hence $\lambda_{11}>\mu_{11}$ and we conclude that 
$$
\mu_{02}\geq\lambda_{11}>\mu_{11}
$$
as claimed.
\qed

\section{Proof of Theorem \ref{barycenter}}\label{sec_barycenter}

We consider a conformal map $\Psi:C_L\to\Omega$. From Lemma \ref{lemma_conformal} we see that
\begin{itemize}
\item $\Psi^\star\psi=\psi_L=L+z$;
\item $\Psi^\star e^{i\Theta}=e^{-i\theta}$;
\item if $u:\Omega\to\mathbb R$ is constant on the level lines of $\psi$, then the function $w:=\Psi^\star u$ depends only on $z$.
\end{itemize}
Hence, for any $u$ which is real and constant on the level lines of $\psi$ we have
$$
\int_{\Omega}ue^{-i\Theta}=\int_{C_L}we^{i\theta}\rho^2=\int_{-L}^Lw(z)\int_0^{2\pi}\rho^2(z,\theta)e^{i\theta}d\theta=0
$$
by the assumption of $\Omega$ being balanced. Now, let $u$ be a minimizer in the definition \eqref{K1-u} of $\kappa_1(\Omega,A)$. It is real and constant on the level lines of $\psi$, and the previous identity shows that it is orthogonal to the first eigenspace of $\Delta_A$, which is spanned by $e^{i\Theta}$ (see Lemma \ref{one}). By the min-max principle:
$$
\lambda_2(\Omega,A)\leq\frac{\int_{\Omega}|d^Au|^2}{\int_{\Omega}|u|^2}=\kappa_1(\Omega,A).
$$
This concludes the proof. \qed

\bibliography{bibliography.bib}
\bibliographystyle{abbrv}

\end{document}